\documentclass[10pt]{amsart}
\usepackage{latexsym,amsmath,amssymb,graphicx,yhmath}
\usepackage{float}
\usepackage[bf, small]{caption}
\usepackage[all,web]{xy}
\usepackage{color}
\usepackage{hyperref}

\usepackage{xcolor}

\def\xyellowspace{%
  \sbox0{\colorbox{yellow}{\strut\ }}%\mathbf{}
  \dimen0=\wd0\relax
  \hskip0pt\cleaders\box0\hskip\dimen0\hskip0pt}

\begingroup
\catcode`\ =\active%
\gdef\makeyellowspace{\let \xyellowspace\catcode`\ =\active}%
\endgroup

\def\?#1{\colorbox{yellow}{\strut#1}}

\DeclareFontFamily{OT1}{rsfs10}{}
\DeclareFontShape{OT1}{rsfs10}{m}{n}{ <-> rsfs10 }{}
\DeclareMathAlphabet{\mathscript}{OT1}{rsfs10}{m}{n}

\DeclareMathOperator{\Cl}{Cl}       % \Cl    = gruppo dei divisori di Weil mod. linear equivalence
\DeclareMathOperator{\Cox}{Cox}     % \Cox   = anello (algebra) di Cox
\DeclareMathOperator{\Mov}{Mov}     % \Mov   = moving cone
\DeclareMathOperator{\Nef}{Nef}     % \Nef   = cono dei divisori nef
\DeclareMathOperator{\Eff}{Eff}     % \Eff   = cono dei divisori effettivi
\DeclareMathOperator{\Relint}{Relint}  % \Relint   = interno di un cono nel suo span lineare
\makeatletter
\def\widebreve{\mathpalette\wide@breve}
\def\wide@breve#1#2{\sbox\z@{$#1#2$}%
     \mathop{\vbox{\m@th\ialign{##\crcr
\kern0.08em\brevefill#1{0.8\wd\z@}\crcr\noalign{\nointerlineskip}%
                    $\hss#1#2\hss$\crcr}}}\limits}
\def\brevefill#1#2{$\m@th\sbox\tw@{$#1($}%
  \hss\resizebox{#2}{\wd\tw@}{\rotatebox[origin=c]{90}{\upshape(}}\hss$}
\makeatletter

\title[On a conjecture of Fujino and Sato]{On a conjecture of Fujino and Sato}

\author[M. Rossi]{Michele Rossi}

\date{\today}

\address{Dipartimento di Matematica, Universit\`a di Milano Bicocca,
Via Roberto Cozzi, 55, 20126 Milano} \email{michele.rossi@unimib.it}

\thanks{The author was partially supported by the I.N.D.A.M. as a member of the G.N.S.A.G.A.\\
 Author's ORCID:0000-0001-6191-2087}
 \keywords{fan, polytope, toric variety, Gale duality, fan matrix, weight matrix, Mori dream space, secondary fan, Galfand-Kaparanov-Zelevinsky decomposition, quasi-isomorphism, small $\Q$-factorial modifications}
\subjclass[2010]{14M25\and 14E30}

\def \s{\sigma }

\def \Si{\Sigma }

\def \g{\gamma}

\def \ét{\'{e}tale}

\def \1{\mathbf{1}}
\def \0{\mathbf{0}}

\def\p2{\mathbb{P}^2}
\def\p3{\mathbb{P}^3}
\def\p4{\mathbb{P}^4}

\def\cO{\mathcal{O}}

\def\K{\mathbb{K}}
\def\R{\mathbb{R}}

\def\Q{\mathbb{Q}}

\def\B{\mathcal{B}}

\def\U1{\mathfrak{U}^{(1)}}

\def\sQm{s\,$\Q$m}

\theoremstyle{plain}
\newtheorem{theorem}{Theorem}[section]

\newtheorem{prop-def}[theorem]{Proposition--Definition}

\newtheorem{thm-def}[theorem]{Theorem--Definition}

\newtheorem{conjecture}[theorem]{Conjecture}

\newtheorem*{a-proposition}{Proposition}

\theoremstyle{remark}
\newtheorem{remark}[theorem]{Remark}

\newtheorem{example}[theorem]{Example}

\theoremstyle{definition}

\newtheorem*{step I}{Step I}
\newtheorem*{step II}{Step II}
\newtheorem*{step III}{Step III}
\newtheorem*{step IV}{Step IV}

\newtheorem*{acknowledgements}{Acknowledgements}

\begin{document}

 %\pagestyle{empty}
 %\DefineParaStyle{Maple Heading 4}
 %\DefineParaStyle{Maple Heading 2}
 %\DefineParaStyle{Maple Text Output}
 %\DefineParaStyle{Maple Bullet Item}
 %\DefineParaStyle{Maple Warning}
 %\DefineParaStyle{Maple Error}
 %\DefineParaStyle{Maple Dash Item}
% \DefineParaStyle{Maple Heading 3}
 %\DefineParaStyle{Maple Heading 1}
 %\DefineParaStyle{Maple Title}
 %\DefineParaStyle{Maple Normal}
% \DefineCharStyle{Maple 2D Input}
% \DefineCharStyle{Maple Maple Input}
 %\DefineCharStyle{Maple 2D Output}
 %\DefineCharStyle{Maple 2D Math}
 %\DefineCharStyle{Maple Hyperlink}
\maketitle

\begin{abstract} We revisit results of Fujino--Sato on complete non-projective $\mathbb Q$-factorial toric varieties and their conjectural factorization by flips. We show that their main results admit short conceptual proofs, avoiding any restriction on the dimension and the Picard number, from the general theory of Cox rings and Mori Dream Spaces, once one organizes small $\mathbb Q$-factorial modifications via the GKZ (secondary fan) decomposition of the moving cone. Moreover, we extend this viewpoint beyond the toric case by proving an analogous statement for complete $\mathbb Q$-factorial weak Mori Dream Spaces: any non-projective such variety admits a divisor $D$ and a $D$-flip to a (projective) Mori Dream Space. Our approach highlights the role of chambers and wall-crossing in the secondary fan as a unifying framework for these constructions.
\end{abstract}

\section*{Introduction}
The birational geometry of a complete variety is often governed by the interplay between divisor cones and small $\mathbb Q$-factorial modifications. In the toric setting this philosophy becomes especially concrete: the movable cone carries a canonical chamber decomposition (the Gelfand-Kapranov-Zelewinsky (GKZ) one or secondary fan) whose chambers encode nef cones of birational models and whose walls govern flips and related wall-crossing operations.

In the recent paper \cite{FS2026}, Fujino and Sato proved that every complete toric variety is birational to a projective $\mathbb Q$-factorial toric variety via a map which is an isomorphism in codimension one, and they formulated a conjecture predicting that any complete $\mathbb Q$-factorial toric variety can be connected to a projective one by a finite sequence of flips, flops and anti-flips (verifying it for smooth threefold of Picard number at most 5). The purpose of the present note is to explain that these results are naturally framed, and in fact streamlined, by the Cox ring/Mori Dream Space machinery, once one adopts the secondary fan viewpoint.

More precisely, our first contribution is a short conceptual derivation of the Fujino--Sato results from the general theory developed by Hu and Keel in their pivotal paper \cite{Hu-Keel} and extended to the complete non-projective setup by Arzhantsev, Derenthal, Hausen and Laface in their extensive treatment \cite{ADHL}, together with the description of small $\mathbb Q$-factorial modifications as wall-crossing in the GKZ decomposition, as described in our previous work \cite{R-wMDS}. In this language, the existence of a projective $\mathbb Q$-factorial model is reflected by the presence of a full-dimensional GKZ chamber contained in the moving cone, while passing from a non-projective situation to a projective one is governed by crossing suitable walls.

Our second contribution is a genuine extension of this picture beyond toric varieties. We prove an analogous statement for complete $\mathbb Q$-factorial \emph{weak Mori Dream Spaces} (wMDS): if $X$ is non-projective, then there exists a divisor $D$ and a $D$-flip producing a (projective) Mori Dream Space. This provides a common conceptual framework for the Fujino--Sato phenomenon and its wMDS generalization, emphasizing that the key mechanism is the chamber structure on the movable cone rather than toric combinatorics per se.

The note is intentionally brief. Its main goal is to isolate a clean conceptual explanation and to point out that the secondary fan viewpoint offers an efficient ``navigation tool'' through birational models and wall-crossing, both in the toric and in the (weak) Mori Dream Space setting.

\begin{acknowledgements}
  I would like to thank H. Sato for the fruitful and stimulating correspondence, and for his encouragement to publish this note.
\end{acknowledgements}

\section{Fujino and Sato results}
In the recent paper \cite{FS2026} O.~Fujino and H.~Sato proved at first the following

\begin{theorem}[Thm.~1.1 in \cite{FS2026}]\label{thm.1.1}
Let $X$ be a complete toric variety. Then we can always construct a projective $\Q$-factorial toric variety $X'$ which is isomorphic to $X$ in codimension 1.
\end{theorem}

Then they stated the following

\begin{conjecture}[Conj.~1.4 in \cite{FS2026}]\label{conj.1.4}
  Let $X$ be a complete $\Q$-factorial toric variety. Then there exists a finite sequence consisting of flips, flops, or anti-flips
  \begin{equation*}
    X=:X_0\dashrightarrow X_1\dashrightarrow\cdots\dashrightarrow X_m=:X'
  \end{equation*}
  such that $X'$ is a projective $\Q$-factorial toric variety.
\end{conjecture}

In their Thm.~1.5 they proved this conjecture for any smooth complete toric threefold of Picard number at most 5.

\section{A short proof of Theorem \ref{thm.1.1}}

First of all, notice that, given a complete toric variety $X(\Si)$, there always exists a small $\Q$-factorial resolution $X'(\Si')$ of $X$, obtained by taking any simplicial subdivision $\Si'$ of the fan $\Si$: namely introduce in any non-simplicial cone  $\s\in\Si$ a sufficient number of faces of dimension $>1$ subdividing $\s$ in a union of simplicial cones. Clearly $\Si'$ turns out to be a refinement of $\Si$ ($\Si'\preceq\Si$) giving rise to an obvious birational morphism $X'\to X$ whose exceptional locus is the union of the closure of the torus orbits of the special points of any new introduced face. Then \emph{$X'$ is isomorphic in codimension 1 to $X$}. Theorem~\ref{thm.1.1} then follows by recalling that a complete $\Q$-factorial toric variety $X'$ is always a weak Mori Dream Space (wMDS), as defined in Definition 7 of \cite{R-wMDS}, and by applying the following

\begin{theorem}[Lemma 1 in \cite{R-wMDS}, Thm.~4.3.3.1 in \cite{ADHL}]\label{lemma1}
  A $\Q$-factorial and complete algebraic variety $X$ is a wMDS if and only if there exists small $\Q$-factorial modification (\sQm) $f:X\dashrightarrow X'$ such that $X'$ is a MDS.
\end{theorem}

\begin{remark}\label{rem}
  Let us recall a few facts from notation introduced in \cite{R-wMDS}.
  \begin{itemize}
    \item[$(i)$] A wMDS is an irreducible and $\Q$-factorial algebraic variety $X$ such that every invertible global function is constant i.e.
    \begin{equation}\label{H0}
      H^0(X,\cO_X)\cong \K^*\,,
    \end{equation}
    its class group $\Cl(X)$ is a finitely generated (f.g.) abelian group of rank
$r := rk(\Cl(X))$ and its Cox ring $\Cox(X)$ is a finitely generated $\K$-algebra. Then $r$ is called the \emph{rank} of $X$ (or the \emph{Picard number}, due to $\Q$-factoriality) and condition (\ref{H0}) is clearly satisfied when X is complete.
    \item[$(ii)$] A projective (hence complete) wMDS is called a Mori Dream Space (MDS), after Hu and Keel \cite{Hu-Keel}.
    \item[$(iii)$] As defined in \cite{Hu-Keel}, a \sQm\ is an isomorphism in codimension 1.
    \item[$(iv)$] A wMDS which is a toric variety $X(\Si)$ is characterized by the fact that its Cox ring is a polynomial algebra
    \begin{equation*}
      \Cox(X(\Si))\cong \K\left[\{x_\rho\}_{\rho\in\Si(1)}\right]
    \end{equation*}
    (see e.g. \cite[Prop.~2]{R-wMDS} and considerations given in the following \cite[Rmk.~2]{R-wMDS}).
    \item[$(v)$] A \sQm\ preserves divisors and so the class group $\Cl(X)$ and the Cox ring $\Cox(X)$.
  \end{itemize}
  Then, in Theorem~\ref{lemma1}, $X$ is toric if and only if $X'$ is toric.
\end{remark}

\subsection{About the structure of the isomorphism in  codimension 1}\label{ssez:sQm} The isomorphism in codimension 1, $\phi:X\dashrightarrow X'$ given by Theorem~\ref{thm.1.1} is a composition $\phi=f\circ\pi^{-1}$, where
\begin{equation*}
  \phi:X\stackrel{\pi^{-1}}{\dashrightarrow}X'\stackrel{f}{\dashrightarrow}X''
\end{equation*}
and $\pi^{-1}$ is the inverse birational map of a small $\Q$-factorial resolution $\pi:X'\to X$, while $f$ turns out to be a $D$-flip (in the sense of Def.~19 in \cite{R-wMDS} and then also in the sense of Def.~2.2 in \cite{FS2026} with respect to any possible small birational morphism $X\to W$ of relative Picard number $\rho(X/W)=1$) for any $\Q$-divisor $D$ whose class is in the relative interior  $\Relint(f^*\Nef(X'')\setminus \Nef(X'))$. In fact:
\begin{itemize}
  \item under our hypotheses, $\overline{\Mov}(X)\cong\overline{\Mov}(X')\cong\overline{\Mov}(X'')\cong\Mov(Q)$ is a full dimensional convex polyhedral cone supporting the fan structure given by the secondary fan (GKZ-subdivision), being $Q$ a \emph{weight matrix} of $X$ (see \S~2.5 in \cite{R-wMDS});
  \item being $X''$ projective, $f^*\Nef(X'')$ is a full dimensional chamber of the secondary fan supported by $\Mov(Q)$ (see Prop.~7 in \cite{R-wMDS} and references therein);
  \item if $X'$ is not projective then $\Nef(X')$ is a face of $f^*\Nef(X'')$ in the secondary fan structure supported by $\Mov(Q)$, so that  $$\Relint(f^*\Nef(X'')\setminus \Nef(X'))\neq\emptyset$$
  \item for any $\Q$-divisor $D$ of $X'$ such that $[D]\in \Relint(f^*\Nef(X'')\setminus \Nef(X'))$ and for any complete curve $C\subset X'$ such that $-mD\cdot C >0$ for some positive integer $m$, then $f(C)$ is a point in $X''$, as the birational transform $f_*(mD):=\overline{f(mD)}$ is ample on $X''$ (see Thm.~6 in \cite{R-wMDS}).
\end{itemize}

\section{The Fujino-Sato Conjecture revisited and proved} The above Conjecture~\ref{conj.1.4} is implied by the following

\begin{theorem}\label{thm}
  Let $X$ be a non-projective complete $\Q$-factorial wMDS. Then, there exists a Cartier divisor $D$ of $X$ and a $D$-flip $f:X\dashrightarrow X'$ such that $X'$ is a MDS. In particular, $X$ is a toric variety if and only if $X'$ is a toric variety.
\end{theorem}

\begin{proof} If $X$ is non-projective then $\Nef(X)$ cannot be a full dimensional chamber of the secondary fan supported by the moving cone $\overline{\Mov(X)}$. Then it is a face of a full dimensional chamber $\g$. By Prop.~9 in \cite{R-wMDS}, there exists a MDS $X'$ and a \sQm\ $f:X\dashrightarrow X'$ such that $\g=f^*\Nef(X')$: actually this is a consequence of Thm.~4.3.3.1 in \cite{ADHL}, which is a generalization to the complete, non-projective setup of Prop.~1.11 in \cite{Hu-Keel}. Considerations given in \S~\ref{ssez:sQm} show that $\mathcal{N}:=\Relint(f^*\Nef(X')\setminus \Nef(X))\neq\emptyset$. Since $X$ is $\Q$-factorial, the choice of a $\Q$-divisor whose class is in $\mathcal{N}$ gives rise, up to a multiple, to a Cartier divisor $D$ such that $[D]\in\mathcal{N}$. Then $f$ is a $D$-flip. What observed in Remark~\ref{rem} shows that $X$ is toric if and only if $X'$ is toric.
\end{proof}

\begin{figure}
\begin{center}
\includegraphics[width=8truecm]{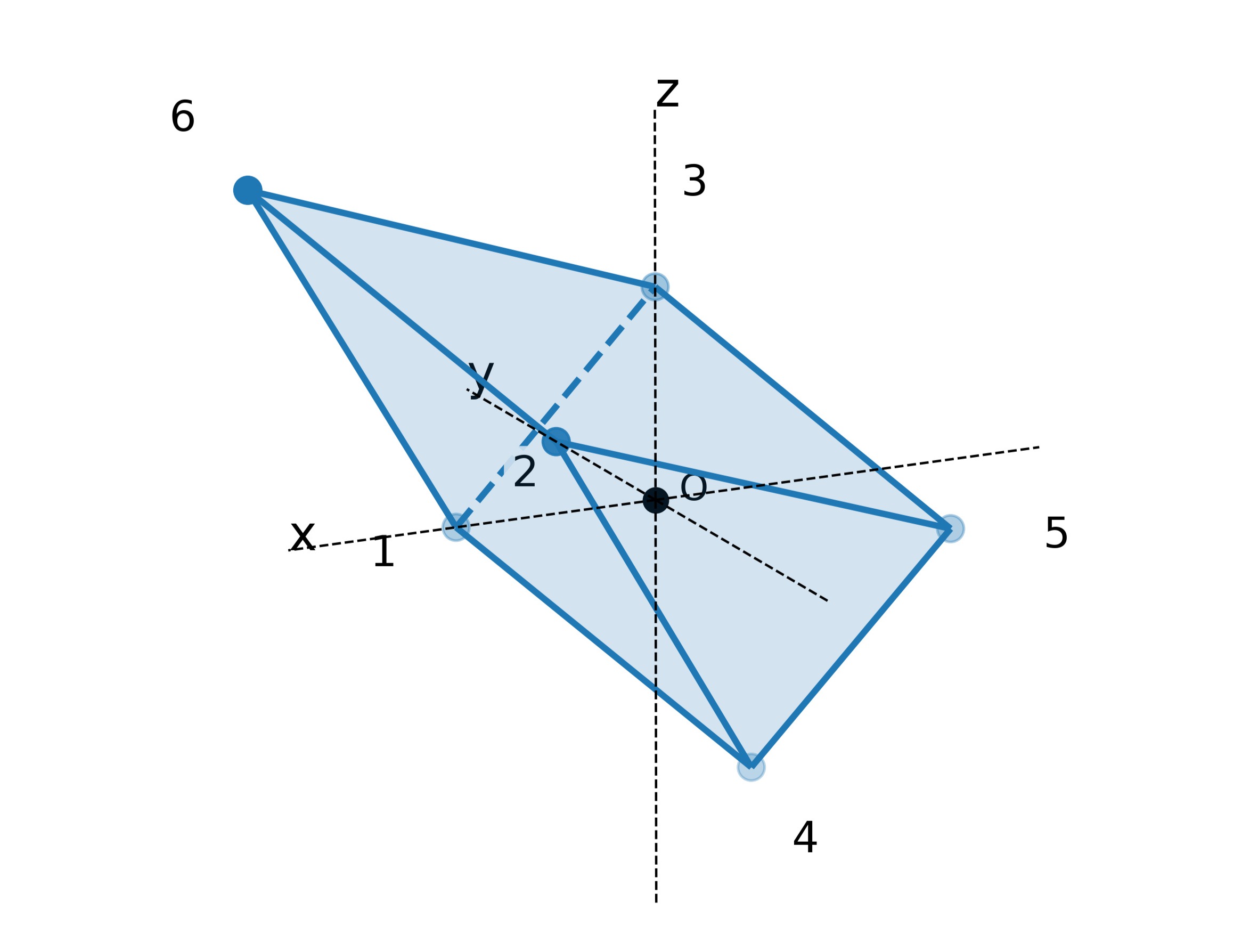}
\caption{\label{Fig1}The prism whose vertices are given by the columns of the fan matrix $V$ in Example~\ref{ex}}
\end{center}
\end{figure}

\begin{example}\label{ex} The present example is devoted to give an account of Theorems~\ref{lemma1} and \ref{thm}, in the toric setup, as considered by Fujino and Sato. We will drop any smoothness hypothesis and consider rank 3 toric varieties to visualize the secondary fan. In any case, this last restriction does not affect the generality of the method in the least.

Consider the prism whose vertex are given in $\R^3$ by the columns of the matrix
\begin{equation*}
  V:= \left(
  \begin{array}{cccccc}
    1&0&0&0&-1&1 \\
    0&1&0&-1&-1&2 \\
    0&0&1&-1&0&1 \\
  \end{array}
\right)
\end{equation*}
and represented in figure \ref{Fig1}, where numbers labeling the vertices are the same of the corresponding column in $V$. Let $\Si$ be the fan spanned by the faces of that prism. The associated toric variety $X(\Si)$ is clearly complete and non-$\Q$-factorial. In particular,  maximal cones of $\Si$ are given by
\begin{equation*}
  \Si(3)=\left\{\langle 1,2,4,6\rangle,\langle 1,3,4,5\rangle,\langle 2,3,5,6\rangle,\langle 2,4,5\rangle,\langle 1,3,6\rangle\right\}
\end{equation*}
where numbers denotes the rays generated by the corresponding column in $V$.

Following the line proving the first part of Fujino-Sato Theorem~1.1 and opening \S~\ref{thm.1.1}, the fan $\Si$ admits $2^3$ simplicial subdivisions, one for each possible subdivision of the 3 quadrangular faces by inserting a 2-dimensional face. This fact gives rise to 8 possible $\Q$-factorial small resolutions $\pi_i:X_i(\Si_i)\to X(\Si)$, for $i=1,\ldots,8$.

Namely they are the complete $\Q$-factorial toric varieties defined by the following simplicial subdivision of $\Si$:
\begin{eqnarray*}
  \Si_1 &=& \{ \langle1, 2, 4\rangle, \langle1, 2, 6\rangle,\langle1, 3, 4\rangle,  \langle3, 4, 5\rangle, \langle2, 3, 5\rangle, \langle2, 3, 6\rangle, \langle2, 4, 5\rangle, \langle1, 3, 6\rangle\} \\
  \Si_2 &=& \{ \langle1, 2, 4\rangle, \langle1, 2, 6\rangle, \langle1, 4, 5\rangle, \langle1, 3, 5\rangle, \langle2, 5, 6\rangle, \langle3, 5, 6\rangle, \langle2, 4, 5\rangle, \langle1, 3, 6\rangle\} \\
  \Si_3 &=& \{\langle1, 4, 6\rangle,  \langle2, 4, 6\rangle, \langle1, 4, 5\rangle, \langle1, 3, 5\rangle, \langle2, 5, 6\rangle, \langle3, 5, 6\rangle, \langle2, 4, 5\rangle,  \langle1, 3, 6\rangle\} \\
  \Si_4 &=& \{\langle1, 2, 4\rangle,  \langle1, 2, 6\rangle,\langle1, 4, 5\rangle, \langle1, 3, 5\rangle, \langle2, 3, 5\rangle,  \langle2, 3, 6\rangle, \langle 2, 4, 5\rangle, \langle1, 3, 6\rangle\} \\
  \Si_5 &=& \{\langle1, 4, 6\rangle, \langle2, 4, 6\rangle, \langle1, 3, 4\rangle, \langle3, 4, 5\rangle, \langle2, 5, 6\rangle,  \langle3, 5, 6\rangle, \langle2, 4, 5\rangle, \langle1, 3, 6\rangle\} \\
  \Si_6 &=& \{\langle1, 4, 6\rangle, \langle2, 4, 6\rangle,  \langle1, 3, 4\rangle, \langle3, 4, 5\rangle,  \langle2, 3, 5\rangle, \langle2, 3, 6\rangle, \langle2, 4, 5\rangle, \langle1, 3, 6\rangle\} \\
  \Si_7 &=& \{\langle1, 2, 4\rangle, \langle1, 2, 6\rangle, \langle1, 3, 4\rangle, \langle3, 4, 5\rangle, \langle2, 5, 6\rangle,  \langle3, 5, 6\rangle, \langle 2, 4, 5\rangle, \langle1, 3, 6\rangle\} \\
  \Si_8 &=& \{\langle1, 4, 6\rangle, \langle2, 4, 6\rangle,\langle1, 4, 5\rangle, \langle1, 3, 5\rangle,  \langle2, 3, 5\rangle,  \langle2, 3, 6\rangle, \langle2, 4, 5\rangle, \langle1, 3, 6\rangle\}
\end{eqnarray*}
All of them are obtained by introducing three 2-dimensional facets, so that every induced birational morphism $\pi_i:X_i\to X$, $i=1,\ldots,8$, is the contraction of 3 exceptional curves.

To consider the secondary fan supported by $\overline{\Mov}(X)$ one has first of all to compute a Gale dual matrix of the fan matrix $V$, for instance given by
\begin{equation*}
  Q=\left(
     \begin{array}{cccccc}
       1&1&0&0&1&0\\
      0&1&1&1&0&0\\
       0&0&0&1&1&1 \\
     \end{array}
   \right)
\end{equation*}
also called a weight matrix of $X$. Then, for any $1\le i\le 8$,
\begin{equation*}
  \overline{\Eff}(X_i)\cong\overline{\Eff}(X)\cong\Eff(Q)=\langle1,3,5\rangle \supset \langle 2,4,5\rangle =\Mov(Q)\cong\overline{\Mov}(X)\cong\overline{\Mov}(X_i)
\end{equation*}
where numbers now refer to the corresponding column of $Q$. The bunches of cones corresponding to fans $\Si_1,\ldots,\Si_8$ are given, respectively by
\begin{eqnarray*}
  \B_1 &=& \{ \langle3,5,6\rangle, \langle3,4,5\rangle,\langle2,5,6\rangle,  \langle1,2,6\rangle, \langle1,4,6\rangle, \langle1,4,5\rangle, \langle1,3,6\rangle, \langle2,4,5\rangle\} \\
  \B_2 &=& \{ \langle3,5,6\rangle, \langle3,4,5\rangle, \langle2,3,6\rangle, \langle2,4,6\rangle, \langle1,3,4\rangle, \langle1,2,4\rangle, \langle1,3,6\rangle, \langle2,4,5\rangle\} \\
  \B_3 &=& \{\langle2,3,5\rangle,  \langle1,3,5\rangle, \langle2,3,6\rangle, \langle2,4,6\rangle, \langle1,3,4\rangle, \langle1,2,4\rangle, \langle1,3,6\rangle,  \langle2,4,5\rangle\} \\
  \B_4 &=& \{\langle3,5,6\rangle,  \langle3,4,5\rangle,\langle2,3,6\rangle, \langle2,4,6\rangle, \langle1,4,6\rangle,  \langle1,4,5\rangle, \langle 1,3,6\rangle, \langle2,4,5\rangle\} \\
  \B_5 &=& \{\langle2,3,5\rangle, \langle1,3,5\rangle, \langle2,5,6\rangle, \langle1,2,6\rangle, \langle1,3,4\rangle,  \langle1,2,4\rangle, \langle1,3,6\rangle, \langle2,4,5\rangle\} \\
  \B_6 &=& \{\langle2,3,5\rangle, \langle1,3,5\rangle,  \langle2,5,6\rangle, \langle1,2,6\rangle,  \langle1,4,6\rangle, \langle1,4,5\rangle, \langle1,3,6\rangle, \langle2,4,5\rangle\} \\
  \B_7 &=& \{\langle3,5,6\rangle, \langle3,4,5\rangle, \langle2,5,6\rangle, \langle1,2,6\rangle, \langle1,3,4\rangle,  \langle1,2,4\rangle, \langle 1,3,6\rangle, \langle2,4,5\rangle\} \\
  \B_8 &=& \{\langle2,3,5\rangle, \langle1,3,5\rangle,\langle2,3,6\rangle, \langle2,4,6\rangle,  \langle1,4,6\rangle,  \langle1,4,5\rangle, \langle1,3,6\rangle, \langle2,4,5\rangle\}
\end{eqnarray*}
The corresponding nef cones are then given by $\Nef(X_i)=\bigcap_{\tau\in\B_i}\tau$, that is,  if $1\le i\le 6$ then $(\pi_i^{-1})^*\Nef(X_i)$ is the full-dimensional chamber $i$ of the secondary fan represented in figure \ref{Fig2}, and, calling $-k:=\left[\sum_{\rho\in\Si(1)}D_\rho\right]$ the anti-canonical class, if $7\le i\le 8$ then $(\pi_1^{-1})^*\Nef(X_i)=\langle -k\rangle$. This means that $\pi_i$ is \emph{projective} $\Q$-factorial resolution of $X$ when $1\le i\le 6$ and a \emph{non-projective} $\Q$-factorial resolution of $X$ when $i=7,8$.

\begin{figure}
\begin{center}
\includegraphics[width=8truecm]{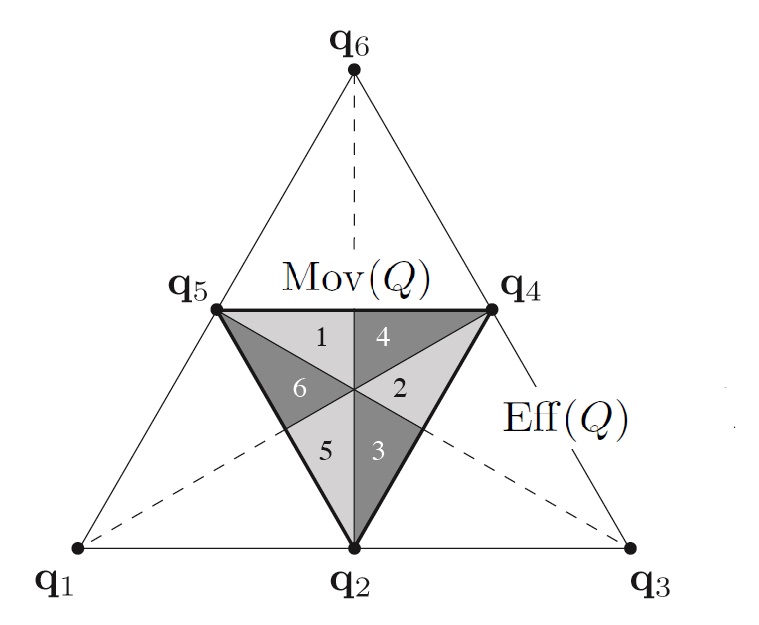}
\caption{\label{Fig2}The section of the cone $\Eff(Q)$ in Example~\ref{ex}, which is the positive orthant of $\R^3$, with the plane $x_1+x_2+x_3=1$.}
\end{center}
\end{figure}

To give an account of the argument proving Theorem~\ref{thm}, notice that the cone generated by the anti-canonical class $\langle -k\rangle$ is a face of any full dimensional chamber of the secondary fan on $\Mov(Q)$, that is, chambers from 1 to 6 in figure \ref{Fig2}. Consider e.g. $\pi_1:X_1\to X$ and the class
\begin{equation*}
  \left(
     \begin{array}{c}
       10 \\
       8 \\
       12 \\
     \end{array}
   \right)\in\Relint(\pi_1^{-1})^*\Nef(X_1)=\Relint\left\langle\left(
                                                                  \begin{array}{c}
                                                                    1 \\
                                                                    0 \\
                                                                    1 \\
                                                                  \end{array}
                                                                \right),\left(
                                                                          \begin{array}{c}
                                                                            1 \\
                                                                            1 \\
                                                                            2 \\
                                                                          \end{array}
                                                                        \right),\left(
                                                                                  \begin{array}{c}
                                                                                    1 \\
                                                                                    1 \\
                                                                                    1 \\
                                                                                  \end{array}
                                                                                \right)
   \right\rangle
\end{equation*}
which is the class of the Cartier ample divisor
$$D=4D_1+2D_2+2D_3+4D_4+4D_5+4D_6$$
Then, $\pi_1^{-1}:X\dashrightarrow X_1$ is a $D$-flip, as $[D]\not\in\Nef(X)=\langle -k\rangle$.
For the same reason, given the $\Q$-factorial complete and non-projective toric variety $X_7$, then
\begin{equation*}
  \pi_1^{-1}\circ\pi_7: X_7\dashrightarrow X_1
\end{equation*}
is a $D$-flip to the projective $\Q$-factorial toric variety $X_1$.

\end{example}

\bibliography{MILEA}
\bibliographystyle{acm}
\end{document}